\documentclass[11pt, a4paper]{amsart}

\usepackage{amsmath}
\usepackage[latin1]{inputenc}
\usepackage[active]{srcltx}

\numberwithin{equation}{section}

 \theoremstyle{plain}
\newtheorem{theorem}[equation]{Theorem}
\newtheorem{lemma}[equation]{Lemma}

\newtheorem{corollary}[equation]{Corollary}

\newtheorem{singular case}[equation]{Proof of Theorem \ref{Holder continuity}}

\theoremstyle{definition}
\newtheorem{definition}[equation]{Definition}

\theoremstyle{remark}
\newtheorem{remark}[equation]{Remark}

\def\kint_#1{\mathchoice%
          {\mathop{\kern 0.2em\vrule width 0.6em height 0.69678ex depth -0.58065ex
                  \kern -0.8em \intop}\nolimits_{\kern -0.4em#1}}%
          {\mathop{\kern 0.1em\vrule width 0.5em height 0.69678ex depth -0.60387ex
                  \kern -0.6em \intop}\nolimits_{#1}}%
          {\mathop{\kern 0.1em\vrule width 0.5em height 0.69678ex depth -0.60387ex
                  \kern -0.6em \intop}\nolimits_{#1}}%
          {\mathop{\kern 0.1em\vrule width 0.5em height 0.69678ex depth -0.60387ex
                  \kern -0.6em \intop}\nolimits_{#1}}}
\def\vintslides_#1{\mathchoice%
          {\mathop{\kern 0.1em\vrule width 0.5em height 0.697ex depth -0.581ex
                  \kern -0.6em \intop}\nolimits_{\kern -0.4em#1}}%
          {\mathop{\kern 0.1em\vrule width 0.3em height 0.697ex depth -0.604ex
                  \kern -0.4em \intop}\nolimits_{#1}}%
          {\mathop{\kern 0.1em\vrule width 0.3em height 0.697ex depth -0.604ex
                  \kern -0.4em \intop}\nolimits_{#1}}%
          {\mathop{\kern 0.1em\vrule width 0.3em height 0.697ex depth -0.604ex
                  \kern -0.4em \intop}\nolimits_{#1}}}

\newcommand{\eps}{\varepsilon}

\newcommand{\R}{\mathbb{R}}
\newcommand{\Rn}{\mathbb{R}^d}

\renewcommand{\limsup}{\operatornamewithlimits{lim \, sup}}

\newcommand{\esssup}{\operatornamewithlimits{ess\, sup}}
\newcommand{\essinf}{\operatornamewithlimits{ess\,inf}}
\newcommand{\essosc}{\operatornamewithlimits{ess\,osc}}

\renewcommand{\l}{\left}
\renewcommand{\r}{\right}

\def\Xint#1{\mathchoice
{\XXint\displaystyle\textstyle{#1}}%
{\XXint\textstyle\scriptstyle{#1}}%
{\XXint\scriptstyle\scriptscriptstyle{#1}}%
{\XXint\scriptscriptstyle\scriptscriptstyle{#1}}%
\!\int}

\def\XXint#1#2#3{{\setbox0=\hbox{$#1{#2#3}{\int}$}
\vcenter{\hbox{$#2#3$}}\kern-.5\wd0}}

\def\dashint{\Xint-}

\title[H\"older regularity for parabolic De Giorgi classes]{H\"older regularity for parabolic De Giorgi classes in metric measure spaces}
\author{Mathias Masson and Juhana Siljander}

\address[M.M]{Aalto University,
Department of Mathematics and Systems Analysis,
P.O. Box 11100
FI-00076 Aalto, Finland}
\email{mathiasmasson@hotmail.com}

\address[J.S.]{University of Helsinki, Department of Mathematics and Statistics, P.O. Box 68, FI-00014 University of Helsinki, Finland}
\email{juhana.siljander@helsinki.fi}

\begin{document}



\begin{abstract}
We give a proof for the H\"older continuity of functions in the parabolic De Giorgi classes in metric measure spaces. We assume the measure to be doubling, to support a weak $(1,p)$-Poincar\'e inequality and to satisfy the annular decay property.
\end{abstract}

\maketitle

\section{Introduction}

The fine properties of the parabolic De Giorgi class are the subject of this paper. This is a class of functions which satisfy a parabolic energy estimate, which in the Euclidean case is related to the $p$-parabolic partial differential equation
\begin{align*}
-\frac{\partial u}{\partial t}+ \textrm{div}(|\nabla u |^{p-2} \nabla u)=0,
\end{align*}
and to compatible parabolic quasiminimizers. A function $u:\Omega \times (0,T) \rightarrow \R$, which is a parabolic $K$-quasiminimizer compatible with the $p$-parabolic partial differential equation in $\Rn$, satisfies
\begin{align*}
 -\int_{\textrm{supp}\,\phi} u \frac{\partial \phi}{\partial t} \, dx\, dt+ &C_1\int_{\textrm{supp}\,\phi} |\nabla u|^p \,dx\, dt\\
&\leq KC_2 \int_{\textrm{supp}\,\phi} |\nabla u - \nabla \phi |^p \, dx\, dt,
\end{align*}
with some $C_1,C_2>0$ and $K\geq 1$, for every smooth compactly supported function $\phi$ in $\Omega \times (0,T)$. Here $\Omega$ denotes a domain in $\Rn$.

Our main result is the local H\"older continuity of the parabolic De Giorgi class functions in metric measure spaces, extending the results obtained for parabolic $K$-quasiminima by Zhou ~\cite{Zhou93, Zhou94} in $\Rn$ with the Lebesgue measure.
Historically, the parabolic version of De Giorgi classes has been investigated in euclidean spaces by Ladyzhenskaja- Solonnikov-Ural'ceva \\~\cite{LadySoloUral68} and DiBenedetto, and later by  ~\cite{DiBe88}, Wieser \cite{Wies87} and Gianazza-Vespri \cite{GianVesp06c}, to name a few. 

Our argument is a modification of the DiBenedetto scheme~\cite{DiBe86, DiBe93, Urba08}. We assume the underlying measure to be doubling and to support a weak $(1,p)$-Poincar\'e inequality. Together these imply a Sobolev inequality. Also, we assume the metric measure space to satisfy the annular decay property \cite{Buck99}. In order to study general measure spaces instead of proving the argument only for the usual Lebesgue measure, we base the proof on integral averages, so that the actual scaling properties of the measure are not needed. A similar technique is applied in ~\cite{KuusLaleSiljUrba10,KuusSiljUrba10}. In the elliptic case, the weighted theory has originally been studied for example by Chiarenza and Serapioni,~\cite{ChiaSera84a, ChiaSera84b, ChiaSera85}.

Due to the general nature of the argument, our approach may open possibilities to establish H\"older continuity in some specific cases of interest, such as for solutions of sub-parabolic equations in Heisenberg groups. Indeed, as the Heisenberg group equipped with the Haar measure is a length space, it is known to satisfy the $\alpha$-annular decay property \cite{Buck99}.

Our argument has several similarities with the Euclidean case, but the new material is substantial. Since standard gradients cannot be defined in a general metric space, we consider upper gradients, and replace the standard Sobolev spaces with Newtonian spaces, see~\cite{BjorBjor11, Shan00}. The parabolic De Giorgi class is defined using the upper gradients, and the modified DiBenedetto method is carried out accordingly.

The motive behind defining the De Giorgi class in this way, is that it enables us to extend the study of parabolic partial differential equations to metric measure spaces. This in turn helps us to better understand those aspects of the theory which are independent of the geometry of the space, where the partial differential equation is originally defined. In practice the extention is done via concepts which in $\R^d$ are closely related to the PDE at hand, and are definable without assuming Euclidean structure of the underlying space. 

For the elliptic case such a concept are the quasiminima, which are known to belong to the elliptic De Giorgi classes, see \cite{Gius03, Giaq93, KinnShan01}. 
 In $\Rn$, classical examples of functions belonging to parabolic De Giorgi classes are the solutions
of parabolic partial differential equations as well as the parabolic quasiminima. The latter has been studied by Zhou in~\cite{Zhou93, Zhou94} and by Wieser in~\cite{Wies87}, whereas for PDEs we refer to~\cite{DiBe93,Urba08}.

In the last part of this paper we show that parabolic quasiminima in  metric measure spaces belong to the parabolic De Giorgi class. A somewhat unexpected difficulty arises in proving the usual De Giorgi estimates for parabolic quasiminima in metric spaces. Indeed, since taking upper gradients is not a linear operation, the usual time mollification argument used in the Euclidean case seems to be destroyed. We circumvent this by introducing so called Cheeger derivatives ~\cite{Chee99}, which have the property that taking a Cheeger derivative is a linear operation. 

\section{Preliminaries} \label{section:preliminaries}
\subsection{Metric measure spaces.} Let $(X,d,\mu)$ be a complete metric measure space with metric $d$ and a positive complete Borel measure $\mu$. The measure $\mu$ is said to be doubling if there exists a universal constant $C_\mu\ge 1$ such that
\[
 \mu(B(x,2r))\le C_\mu \mu(B(x,r)),
\]
for every $r>0$ and $x\in X$. Here $B(x,r)$ denotes the standard open ball
\[
B(x,r)=\{y\in X : d(x,y)<r\}.
\]
The dimension related to the doubling measure is defined to be 
\begin{align*}
 d_\mu=\log_2 C_\mu.
\end{align*}
 By iterating the doubling condition, it follows that
\begin{equation*} \label{eq:DoublingConsequence}
\frac{\mu(B(z,r))}{\mu(B(y,R))} \geq C_\mu^{-2}\Bigl(\frac{r}{R}\Bigr)^{d_\mu},
\end{equation*}
for all balls $B(y,R)\subset X$, $z \in B(y,R)$
and $0 < r \leq R < \infty$.
Given $\alpha>0$ and a metric space $(X,d,\mu)$ with a doubling $\mu$, we say that the space satisfies the $\alpha$-annular decay property if there exists a constant $c\ge 1$ such that
\begin{equation}\label{annular_decay}
\mu(B(x,r)\setminus B(x,(1-\delta)r))\le c \delta^\alpha\mu(B(x,r))
\end{equation}
for all $B(x, r)\subset X$ and for all
 $0<\delta<1$. Every length space has this property, in particular this is true in $\Rn$. For further information about the spaces which satisfy this, see \cite{Buck99}.

\subsection{Upper gradients.} Let $\Omega\subset X$ be open. Following \cite{HeinKosk98}, a non-negative Borel measurable function $g: \Omega \rightarrow [0, \infty]$ is said to be an upper gradient of a function $u: \Omega\rightarrow [-\infty, \infty]$ in $\Omega$, if for all compact rectifiable paths $\gamma$ joining $x$ and $y$ in $\Omega$ we have 
\begin{align}\label{upper gradient}
|u(x)-u(y)|\leq \int_{\gamma} g \,ds.
\end{align}
In case $u(x)=u(y)=\infty$ or $u(x)=u(y)=-\infty$, the left side is defined to be $\infty$. Assume $1\leq p <\infty$. The $p$-modulus of a family of paths $\Gamma$ in $\Omega$ is defined to be
\begin{align*}
 \inf_\rho \int_\Omega \rho^p \,d\mu,
\end{align*}
where the infimum is taken over all non-negative Borel measurable functions $\rho$ such that for all rectifiable paths $\gamma$ which belong to $\Gamma$, we have
\begin{align*}
 \int_{\gamma} \rho \, ds \geq 1.
\end{align*}
A property is said to hold for $p$-almost all paths, if the set of non-constant paths for which the property fails is of zero $p$-modulus. Following \cite{KoskMacm98,Shan00}, if \eqref{upper gradient} holds for $p$-almost all paths $\gamma$ in $X$, then $g$ is said to be a $p$-weak upper gradient of $u$. 

When $1<p<\infty$ and $u\in L^p(\Omega)$, it can be shown \cite{Shan01} that  there exists a minimal $p$-weak upper gradient of $u$, we denote it by $g_u$, in the sense that $g_u$ is a $p$-weak upper gradient of $u$
and for every $p$-weak upper gradient $g$ of $u$ it holds $g_u\leq g$ $\mu$-almost everywhere in $\Omega$. 
Moreover, if $v=u$ $\mu$-almost everywhere in a Borel set $A\subset \Omega$, then $g_v=g_u$ $\mu$-almost everywhere in $A$. Also, if $u,v\in L^p(\Omega)$, then $\mu$-almost everywhere in $\Omega$, we have
\begin{align*}
&g_{u+v}\leq g_u+g_v,\\
&g_{uv}\leq |u|g_v+|v|g_u.
\end{align*}
Proofs for these properties and more on upper gradients in metric spaces can be found for example in \cite{BjorBjor11} and the references therein.


\subsection{Newtonian spaces.} Following \cite{Shan00}, for $1< p<\infty$ and $u\in L^p(\Omega)$, we define
\begin{align*}
\|u\|_{1,p,\Omega}=\|u\|_{L^p(\Omega,\mu)}+\|g_u\|_{L^p(\Omega,\mu)},
\end{align*}
and
\begin{align*}
\widetilde N^{1.p}(\Omega)= \{ u\,:\, \|u\|_{1,p,\Omega}<\infty\}.
\end{align*}
An equivalence relation in $\widetilde N^{1,p}(\Omega)$ is defined by saying that $u\sim v$ if
\begin{align*}
 \|u-v\|_{\widetilde N^{1,p}(\Omega)}=0.
\end{align*}
The \emph{Newtonian space} $N^{1,p}(\Omega)$ is defined to be the space $\widetilde N^{1,p}(\Omega)/ \sim$, with the norm
\begin{align*}
 \|u\|_{N^{1,p}(\Omega)}=\|u\|_{1,p,\Omega}.
\end{align*}

A function $u$ belongs to the local Newtonian space $N_{\textrm{loc}}^{1,p}(\Omega)$ if it belongs to $N^{1,p}(\Omega')$ for every $\Omega' \subset \subset \Omega$. For more properties of Newtonian spaces, see \cite{Hein01,Shan00, BjorBjor11}.

\subsection{Poincar\'e's inequality}
For positive $1\leq q <\infty$, $1\leq p < \infty$, the measure $\mu$ is said to support a weak $(q,p)$-Poincar\'e inequality if there exist constants $P_0>0$ and $\tau\ge 1$ such that
\begin{equation}
\l(\dashint_{B(x,r)}|v-v_{B(x,r)}|^q \, d\mu\r)^{1/q} \le P_0 r\left(\dashint_{B(x,\tau r)} g_v^p \, d\mu\right)^{1/p},
\end{equation}
for every $v\in N^{1,p}(X)$ and $B(x,\tau r)\subset X$. Here we use the notation
\[
 v_{B(x,r)}=\dashint_{B(x,r)} v \, d\mu = \frac{1}{\mu(B(x,r))}\int_{B(x,r)} v \, d\mu.
\]

In case $\tau=1$, we say a $(q,p)$-Poincaré inequality is in force. In a general metric measure space setting, it is of interest to have assumptions which are invariant under bi-Lipschitz mappings. The weak $(q,p)$-Poincaré inequality has this quality. 

For a metric space $X$ equipped with a doubling measure $\mu$, the following result of \cite{HajlKosk95} is known: If  $X$ supports a weak $(1,p)$-Poincaré inequality for some $1<p<\infty$, then $X$ also supports a weak $(\kappa,p)$-Poincaré
inequality, where
\begin{equation*}
\kappa=\begin{cases}\frac{d_\mu p}{d_\mu-p},  &\text{for} \quad 1<p<d_\mu, \\ 2p, &\text{otherwise}, \end{cases}
\end{equation*}
possibly with different constants $P_0'>0$ and $\tau'\geq 1$. As a consequence of this, by the $(\kappa,p)$-Poincaré inequality and for example by Poposition 5.41 in \cite{BjorBjor11}, there exists a positive constant $C$  such that 
\begin{align}\label{Sobolev}
\left(\dashint_{B(x,r)} |v|^\kappa \, d\mu\right)^{1/\kappa}\le Cr\left(\dashint_{B(x,r)} g_v^p\, d\mu\right)^{1/p}.
\end{align}
for every $v\in N_0^{1,p}(X)$ and $B(x,r)\subset X$.

\begin{remark} \label{poincare_remark}
It is a result of \cite{KeitZhon08}, that when $1<p<\infty$ and $(X,d)$ is a complete metric space with doubling measure $\mu$, the weak $(1,p)$-Poincaré inequality implies a weak $(1,q)$-Poincar\'e inequality for some $1<q<p$. Then by the above discussion, $X$ also supports a weak $(\kappa,q)$-Poincaré inequality with some $\kappa>q$. By Hölder's inequality, the left hand side of the weak $(\kappa,q)$-Poincaré inequality can be estimated from below by replacing $\kappa$ with any positive $\kappa'<\kappa$. Hence we conclude, that if $X$ supports a weak $(1,p)$-Poincar\'e inequality with $1<p<\infty$, then $X$ also supports a weak $(q,q)$-Poincar\'e inequality with some $1<q<p$.

\end{remark}

\subsection{Parabolic upper gradients and Newtonian spaces}
We define the parabolic Newtonian space  $L^p(0,T;N^{1,p}(\Omega))$ to be the space of functions $u(x,t)$ such that for almost every $0<t<T$ the function $u(\cdot, t)$ belongs to $N^{1,p}(\Omega)$, and
\[
\int\limits_{0}^{T}\|u(\cdot,t)\|_{1,p,\Omega}^p\, dt < \infty.
\]
The definition of the space $L_{\textrm{loc}}^p(0,T;N_{\textrm{loc}}^{1,p}(\Omega))$ is obvious. In what follows we will denote the product measure by $d\nu=d\mu\, dt$.
Let $u\in L_{\textrm{loc}}^p(0,T;N_{\textrm{loc}}^{1,p}(\Omega))$. The parabolic minimal $p$-weak upper gradient of $u$ is defined in a natural way by setting
\begin{align*}
g_u(x,t)=g_{u(\cdot,t)}(x), 
\end{align*}
at $\nu$-almost every $(x,t)\in \Omega\times (0,T)=\Omega_T$. For the sake of conciseness we refer to the parabolic minimal $p$-weak upper gradient of a time dependent function, by calling it  the upper gradient. 

Next we define the class of functions for which we prove local Hölder continuity.
\begin{definition}[Parabolic De Giorgi class]
\label{De Giorgi class}
For $p>2$, we say that $u\in L_\textrm{loc}^p(0,T;N_{\textrm{loc}}^{1,p}(\Omega))$ belongs to the \emph{parabolic De Giorgi class}, if there exists a positive constant $C$, with which we have 
\begin{align*}
&\esssup_{\tau_1<t<\tau_0} \int_{B(x,r_1)} (u(x,t)-k)_\pm^2 \, d\mu + \int_{\tau_1}^{\tau_0} \int_{B(x,r_1)} g_{(u-k)_\pm}^p \, d\nu\\
&\leq  \frac{C}{(r_2-r_1)^p}\int_{\tau_2}^{\tau_0}\int_{B(x,r_2)} (u-k)_\pm^p \, d\nu+\frac{C}{(\tau_1-\tau_2)}\int_{\tau_2}^{\tau_1}\int_{B(x,r_2)}(u-k)_\pm^2 \, d\nu,
\end{align*}
for every $x\in \Omega$,  $r_1<r_2$  and $\tau_2<\tau_1<\tau_0$, such that $B(x,r_2)\times (\tau_2,\tau_0)\subset \Omega_T$. As an immediate consequence, by the Lebesgue differentiation theorem, we see by taking the limit $\tau_2\rightarrow \tau_1$, that for almost every $\tau_1 \in (0,\tau_0)$, we have
\begin{equation}\label{DGC limit}
\begin{split}
&\esssup_{\tau_1<t<\tau_0} \int_{B(x,r_1)} (u(x,t)-k)_\pm^2 \, d\mu + \int_{\tau_1}^{\tau_0} \int_{B(x,r_1)} g_{(u-k)_\pm}^p \, d\nu\\
&\leq  \frac{C}{(r_2-r_1)^p}\int_{\tau_1}^{\tau_0}\int_{B(x,r_2)} (u-k)_\pm^p \, d\nu+C\int_{B(x,r_2)}(u(x,\tau_1)-k)_\pm^2 \, d\mu.
\end{split}
\end{equation}
\end{definition}
The following regularity theorem is the main result of this paper.
\begin{theorem}\label{Holder continuity}
Assume $(X, d)$ is a complete metric space equipped with a complete doubling positive Borel measure $\mu$ . Assume $X$ supports a weak $(1,p)$-Poincaré inequality for  $2<p<\infty$, and satisfies the $\alpha$-annular decay property for some $0<\alpha <1$. Let $u\in  L_{\textrm{loc}}^p(0,T;N_{\textrm{loc}}^{1,p}(\Omega))$ be locally essentially bounded and belong to the parabolic De Giorgi class. Then $u$ has a representative which is locally H\"older continuous in $\Omega_T$.
\end{theorem}

The proof is based on the DiBenedetto argument and it is divided into two alternatives. We will use intrinsic scaling \cite{Urba08} in order to homogenize the powers on the right hand side of the estimate in definition~\ref{De Giorgi class}. We start by constructing the geometry.

\subsection{Initial setting}\label{constructing the argument}

Let $\lambda\geq 1$ be an a priori constant, which has been fixed in a manner which will become clear later in this text. Assume a compact set $S \subset \Omega_T$. Our aim is to show that $u$ is Hölder continuous in $S$. Since $S$ is compact, there exists an open set $F$ such that $S\subset F\subset \subset \Omega_T$.

Since $u$ is assumed to be locally essentially bounded, $\essosc_F u$ is finite. We assume that $\essosc_F u>0$, because otherwise Hölder continuity in $S$ is trivially true. We redefine $u(x,t)=\esssup_F u$ in the $\nu$-negligible subset of $F$ where $u(x,t)\not \in[\essinf_F u, \esssup_F u]$. Set
\begin{align*}
\theta_\pm=\left(\frac{\essosc_F u}{\gamma_\pm} \right)^{2-p},
\end{align*}
where $\gamma_-=2$ and $\gamma_+=2^\lambda$. Let $r$ be a positive number such that for every $(x,t)\in S$ we have $B(x,2\tau r)\times (t-2\theta_+ r^p,t) \subset F$. Clearly $r$ is controlled by the distance of $S$ to the complement of $\Omega_T$.

Considering any point $(x,t)\in S$, it will turn out that the constants in the reduction of oscillation of $u$ in the neighborhood of $(x,t)$ depend only on $p,d_\mu,C_\mu,P_0,\tau$ and on the constant $r$. This in turn implies by a standard iterative argument \cite{Urba08}, that $u$ is Hölder continuous in $S$, and that the modulus of Hölder continuity will depend only on $p,d_\mu,C_\mu,P_0, \tau, r$ and on $\essosc_F u$. Therefore, to prove the Hölder continuity of $u$ in a compact set $S$, it is enough to consider any one point in $S$, and examine the oscillation of $u$ in its neighborhood.

So from here on, let $(x_0,t_0)$ denote some fixed point in $S$. For brevity, in what follows we will refer to the set of constants $p,d_\mu,C_\mu,P_0, \tau$ as the data.

\section{Estimates for the parabolic De Giorgi class}

Let $t^*\in (t_0-\theta_+r^p,t_0)$. For $n=0,1,\dots$ we denote 
\begin{align*}
r_n=\frac{r_0}{2}+\frac{r_0}{2^{n+1}}, \qquad Q_n^\pm=B_n\times T_n^\pm =B(x_0,r_n)\times(t^*-\theta_\pm r_n^p,t^*),
\end{align*}
and
\begin{align*}
 A_n^\pm=\left\{(x,t)\in Q_n^\pm: (u(x,t)-k_n^\pm)_\pm> 0\right\},
\end{align*}
where $r_0$ and $k_n^\pm$ will be chosen later, according to the situation at hand. The following two lemmas are of central importance to the proof.
\begin{lemma}\label{main_lemma}
Let $(k_n^+)_n$ be an increasing sequence and $(k_n^-)_n$ a decreasing sequence of real numbers.
Let $u$ belong to the parabolic De Giorgi class and assume that for some number $0<\eps<1$, we have
\begin{equation}\label{eps_bounds}
(u-k_n^\pm)_\pm\le \eps\essosc_F{u} \qquad\text{and}\qquad
|k_{n+1}^\pm-k_n^\pm|\ge\frac{\eps\essosc_F{u}}{2^{n+2}},
\end{equation}
Then there exists a constant $C_0$ which depends only on the data such that
\begin{equation*}
\begin{split}
\frac{\nu(A_{n+2}^\pm)}{\nu(Q_{n+2}^\pm)}\le C_0(3+(\varepsilon \gamma_\pm)^{p-2}+(\varepsilon \gamma_\pm)^{2-p})^{2-p/\kappa}4^{pn(2-p/\kappa)}\left(\frac{\nu(A_{n}^\pm)}{\nu(Q_{n}^\pm)}\right)^{2-p/\kappa}
\end{split}
\end{equation*}
for every $n=0,1,\dots$.

\begin{proof}
For each $n$, let $\varphi_n \in \textrm{Lip}(\Omega)$, $0\leq \varphi_n \leq 1$, be a cut off function such that $\varphi_n(x)=1$ in $B_{n+1}$, the support of $\varphi_n$ is a compact subset of $B_n$, and $g_{\varphi_n}\leq 2^{n+2}/r_0$. Let $\kappa$ be the constant from \eqref{Sobolev}. By Hölder's inequality we can write
\begin{equation*}
\begin{split}
&\dashint_{Q_{n+2}^\pm} (u-k_n^\pm)_\pm^{p(2-p/\kappa)} d \nu\leq \dashint_{T_{n+2}^\pm} \left( \dashint_{B_{n+2}} (u-k_n^\pm)^p_{\pm} d\mu \right)^{(\kappa-p)/ \kappa} \\
& \qquad \times \left( \dashint_{B_{n+2}} (u-k_n^\pm)_\pm ^\kappa d\mu \right)^{p/\kappa} dt \\
&\leq \frac{|T_{n+1}^\pm| \mu (B_{n+1}) }{ |T_{n+2}^\pm| \mu (B_{n+2})} \left( \esssup_{T_{n+1}^\pm} \dashint_{B_{n+1}}(u-k_n^\pm)^p_\pm  d\mu \right)^{(\kappa-p)/\kappa}\\
&\qquad\times \dashint_{T_{n+1}^\pm} \left( \dashint_{B_{n+1}} ((u-k_n^\pm)_\pm \varphi_{n+1})^\kappa d\mu \right)^{p/\kappa} \, dt. \\
\end{split}
\end{equation*}
By the doubling property of $\mu$, the measure factor on the right hand side is uniformly bounded for every $n$.  Now, since $(u-k_n^\pm)_\pm \varphi_{n+1}\in N_0^{1,p}(B_{n+1})$, we can use inequality \eqref{Sobolev} to obtain 
\begin{equation*}\label{main lemman keskiosa}
\begin{split}
&\dashint_{T_{n+1}^\pm} \l(\dashint_{B_{n+1}} ((u-k_n^\pm)_\pm \varphi_{n+1})^\kappa d\mu \right)^{p/\kappa}\,dt \le C r_0^p \dashint_{Q^\pm_{n+1}} g_{(u-k_n^\pm)_\pm \varphi_{n+1}}^p \, d\nu \\
&\leq Cr_0^p  \dashint_{Q^\pm_{n+1}} g_{(u-k_n^\pm)_\pm}^p\,d\nu +C2^{p(n+2)}\dashint_{Q^\pm_{n+1}} (u-k_n^\pm)_\pm^p \,d\nu\\
&\leq Cr_0^p  \dashint_{Q^\pm_{n+1}} g_{(u-k_n^\pm)_\pm}^p\,d\nu+C2^{p(n+2)}(\eps\essosc_F u)^p\frac{\nu(A_n^\pm)}{\nu(Q_n^\pm)}.
\end{split}
\end{equation*}

Next we use the definition of the parabolic De Giorgi class to estimate the first integral term on the right hand side of the above expression. We have
\begin{align*}
&r_0^p\dashint_{Q^\pm_{n+1}} g_{(u-k_n^\pm)_\pm}^p\,d\nu 
\leq C\left( 2^{pn}\dashint_{Q^\pm_n} (u-k_n^\pm)_\pm^p \,d\nu+ \frac{2^{pn}}{\theta_\pm }\dashint_{Q^\pm_n} (u-k_n^\pm)_\pm^2 \,d\nu\right)\\
& \leq C2^{pn}\left(1+\frac{1}{\theta_\pm}(\eps\essosc_F u)^{2-p}\right)(\eps\essosc_F u)^p \frac{\nu(A^\pm_n)}{\nu(Q^\pm_n)}\\
&= C2^{pn} \l(1+\l(\gamma_\pm\eps\r)^{2-p}\r)(\eps\essosc_F u)^p \frac{\nu(A^\pm_n)}{\nu(Q^\pm_n)}.
\end{align*}
Similarly, since $u$ belongs to the parabolic De Giorgi class, by the doubling property of $\mu$ and by the fact that $\mu(B_{n})\theta_\pm r_n^p=\nu (Q_n^\pm)$, we have
\begin{align*}
&\esssup_{T_{n+1}^\pm}\dashint_{B_{n+1}}  (u-k_n^\pm)_\pm^p \, d\mu\le\frac{(\gamma_\pm\eps)^{p-2}}{\theta_\pm}\esssup_{T_{n+1}^\pm}\dashint_{B_{n+1}}  (u-k_n^\pm)_\pm^2 \, d\mu\\
&\leq C(\gamma_\pm\eps)^{p-2}\left(2^{pn} \dashint_{Q^\pm_n} (u-k_n^\pm)_\pm^p \,d\nu+ \frac{2^{pn}}{\theta_\pm }\dashint_{Q^\pm_n} (u-k_n^\pm)_\pm^2 \,d\nu\right)\\
&\leq C 2^{pn}  \l((\gamma_\pm\eps)^{p-2}+1\r)(\eps\essosc_F u)^p \frac{\nu(A^\pm_n)}{\nu(Q^\pm_n)}.
\end{align*} 
Collecting the obtained estimates yields, 
\begin{align*}
&\dashint_{Q^\pm_{n+2}} (u-k_n^\pm)_\pm^{p(2-p/\kappa)} d \nu \\
&\qquad\le \l(C 2^{pn}(3+(\varepsilon \gamma_\pm)^{p-2}+(\varepsilon \gamma_\pm)^{2-p})(\eps\essosc_F u)^p \frac{\nu(A^\pm_n)}{\nu(Q^\pm_n)}\right)^{2-p/\kappa}.
\end{align*}
On the other hand, by~\eqref{eps_bounds} and since $\mu$ is doubling, we have
\begin{align*}
 &\dashint_{Q_{n+2}^\pm} (u-k_n^\pm)_\pm^{p(2-p/\kappa)}\, d\nu \geq |k_{n+1}^\pm-k_n^\pm|^{p(2-p/\kappa)}\frac{\nu(A_{n+2}^\pm)}{\nu(Q_{n+2}^\pm)}\\
&\quad\geq \left(\frac{\eps\essosc_F u}{2^{n+2}}\right)^{p(2-p/\kappa)} \frac{\nu(A_{n+2}^\pm)}{\nu(Q_{n+2}^\pm)}.
\end{align*}
Therefore, for each $n\geq 0$ we have
\begin{align*}
\frac{\nu(A_{n+2}^\pm)}{\nu(Q_{n+2}^\pm)}\le C4^{pn(2-p/\kappa)} (3+(\eps\gamma_\pm)^{p-2}+(\varepsilon \gamma_\pm)^{2-p})^{2-p/\kappa}\left(\frac{\nu(A_{n}^\pm)}{\nu(Q_{n}^\pm)}\right)^{2-p/\kappa},
\end{align*}
where $C$ depends only on the data.
\end{proof}
\end{lemma}
We also prove the following time independent variant of the above lemma.
\begin{lemma}\label{time-independent main lemma} Suppose the assumptions of Lemma \ref{main_lemma} hold. Suppose in addition that for some $t_0-2\theta_+ r_0^p<t'<t_0$, where $t'$ is a Lebesgue point of the mapping \begin{align*}
t\mapsto \int_{B(x_0,r)}(u(x,t)-k)_\pm^2\,d\mu,
\end{align*}
we have $(u-k_n^\pm)_\pm(x,t')=0$ at $\mu$-almost every $x\in B(x_0,r_n)$ for every $n\geq0$. Then for the time independent sequence 
\begin{align*}
 Q_n^\pm=B(x_0,r_n)\times(t',t_0),
\end{align*}
and corresponding sets $A_n^\pm$,
there exists a constant $C_0$, which depends only on the data, such that
\begin{align*}\label{corollaryestimate}
\frac{\nu(A_{n+2}^\pm)}{\nu(Q_{n+2}^\pm)}\le C_0\l(2+2^{(\lambda+1)(p-2)}\r)^{2-p/\kappa}4^{pn(2-p/\kappa)}\left(\frac{\nu(A_{n}^\pm)}{\nu(Q_{n}^\pm)}\right)^{2-p/\kappa}
\end{align*}
for every $n=0,1,2,\dots$.
\begin{proof}
We proceed exactly as in the proof of Lemma \ref{main_lemma}, with $T_n^\pm=(t',t_0)$ for every $n$, but now instead of the parabolic De Giorgi class, we can use \eqref{DGC limit} and the assumption that we have $(u-k_n^\pm)_\pm(x,t')=0$ for $\mu$-almost every $x\in B(x_0,r_0)$. We estimate
\begin{align*}
&r_0^p\dashint_{Q^\pm_{n+1}} g_{(u-k_n^\pm)_\pm}^p\,d\nu 
\leq C2^{pn}\dashint_{Q^\pm_n} (u-k_n^\pm)_\pm^p \,d\nu\leq  C2^{pn}(\eps\essosc_F u)^p \frac{\nu(A^\pm_n)}{\nu(Q^\pm_n)},
\end{align*}
and
\begin{align*}
&\esssup_{T_{n+1}^\pm}\dashint_{B_{n+1}}  (u-k_n^\pm)_\pm  ^p \, d\mu  
\leq C(\eps\essosc_F u)^{p-2}\frac{2^{pn} (t_0-t')}{r_0^p}\dashint_{Q^\pm_n} (u-k_n^\pm)_\pm^p \,d\nu\\
&\quad\leq C 2^{pn}  (\eps 2\gamma_+)^{p-2}(\eps\essosc_F u)^p \frac{\nu(A^\pm_n)}{\nu(Q^\pm_n)}.
\end{align*} 
Plugging these into the proof of Lemma \ref{main_lemma} yields
\begin{align*}
&\dashint_{Q^\pm_{n+2}} (u-k_n^\pm)_\pm^{p(2-p/\kappa)} d \nu \\
&\quad\le \left(C 2^{pn}\l(2+(\varepsilon 2\gamma_+)^{p-2}\r) (\eps\essosc_F u)^{p}  \frac{\nu(A^\pm_n)}{\nu(Q^\pm_n)}\right)^{2-p/\kappa}.
\end{align*}
As in the proof of Lemma \ref{main_lemma}, after noting that $\eps<1$, this leads to the estimate
\begin{align*}
\frac{\nu(A_{n+2}^\pm)}{\nu(Q_{n+2}^\pm)}\le C4^{pn(2-p/\kappa)} \l(2+2^{(\lambda+1)(p-2)}\r)^{2-p/\kappa}\left(\frac{\nu(A_{n}^\pm)}{\nu(Q_{n}^\pm)}\right)^{2-p/\kappa}.
\end{align*}
\end{proof}
\end{lemma}
Later in this proof, depending on the situation, we will use Lemma \ref{main_lemma} or Lemma \ref{time-independent main lemma} together with the following real analytic lemma.
\begin{lemma}\label{geometric_convergence}
Let $(Y_n)_n$ be a sequence of positive numbers, satisfying
\begin{equation*}\label{rec-ineq}
Y_{n+1}\le Cb^nY_n^{1+\alpha}
\end{equation*}
where $C,b>1$ and $\alpha>0$. Then $Y_n$ converges to $0$ as $n\rightarrow\infty$, provided that
\begin{equation}\label{strt-engine}
Y_0\le C^{-1/\alpha}b^{1-\alpha^2}.
\end{equation}
\begin{proof} 
For the proof, we refer to~\cite{DiBe93}.
\end{proof}
\end{lemma}
By Lemma \ref{geometric_convergence}, once the requirements of Lemma \ref{main_lemma} or Lemma \ref{time-independent main lemma} have been established, the convergence to zero of $\nu(A_{2n}^\pm)$ follows, provided we can first show that the corresponding initial condition \eqref{strt-engine} is satisfied.

Next we divide the proof in two complementary alternatives and study them separately.

\section{The First Alternative}

Recall that $\gamma_-=2$, and set
\begin{equation}\label{k_n^-}
 k_n^-=\essinf_{F} u +\frac{\essosc_{F} u}{4}+\frac{\essosc_{F} u}{2^{n+2}}\quad\text{and}\quad  r_0=r.
\end{equation}
Then
\begin{equation*}
(u-k_n^-)_-\le \frac{\essosc_F{u}}{2} \qquad\text{and}\qquad
|k_{n+1}^\pm-k_n^\pm|\ge\frac{1}{2}\frac{\essosc_F{u}}{2^{n+2}},
\end{equation*}
and so we see that here $\varepsilon=1/2$. Plug these into Lemma \ref{main_lemma}, and define
\begin{align*}
\alpha_0=(C_0(3+1^{p-2}+1^{2-p})^{2-p/\kappa})^{-1/(1-p/\kappa)}(4^{p2(2-p/\kappa)})^{1-(1-p/\kappa)^2}.
\end{align*}
Note that $\alpha_0$ does not depend on $\lambda$.

Suppose there exists a $t^*\in(t_0-\theta_+ r^p,t_0)$ such that
\begin{equation}\label{first_alternative_assumption}
\begin{split}
&\nu\left(\left\{B(x_0,r)\times(t^*-\theta_- r^p,t^*) :u(x,t)\leq k_0^-\right\}\right) \\
&\qquad< \alpha_0 \nu(B(x_0,r)\times(t^*-\theta_- r^p,t^*)).
\end{split} 
\end{equation}
We refer to this condition as \emph{the first alternative}. Using the abbreviations introduced in the previous section we can write inequality \eqref{first_alternative_assumption} as
\begin{equation*}\label{abbreviated first alternative}
\frac{\nu(A_0^-)}{\nu(Q_0^-)}< \alpha_0.
\end{equation*}
By Lemma \ref{main_lemma} and Lemma~\ref{geometric_convergence} this implies that $\nu(A_{2n}^-)/\nu(Q_{2n}^-)\rightarrow 0$
as $n\rightarrow 0$.  Hence if the first alternative is in force, we obtain
\begin{align*}
&u(x,t)\ge\essinf_{F} u +\frac{\essosc_{F} u}{4}
\end{align*}
$\nu$-a.e. in $B\l(x_0,\frac{r}{2}\r)\times\left(t^*-\theta_-\left(\frac{r}{2}\right)^p,t^*\right)$.
This implies that we can fix a time level $
t' \in \l(t_0-2\theta_+r^p, t_0-\theta_- \left(\frac{r}{4} \right)^p\r)$
at which we have
\begin{align}\label{linecoldness}
 u(x,t')\ge\essinf_{F} u +\frac{\essosc_{F} u}{4}\text{ for $\mu$-a.e.} \quad x\in B\l(x_0,r/2\r).
\end{align}
Now we are able to use \eqref{DGC limit} to obtain the following.
\begin{lemma}\label{initial_estimate}
Assume the first alternative.
 Then there exists a constant $C$, which depends only on the data, such that for all $s>1$ we have
\[
\frac{\nu(\{\,B\l(x_0,r/4\r)\times(t',t_0):\,u<\essinf_F u+ \frac{\essosc_F u}{2^{\lambda + s}}\,\} )}{ \nu(B(x_0,r/4)\times(t',t_0))}\leq \frac{C}{2^{s(p-2)}}.
\]

 \begin{proof}
  


We set
\begin{align*}
k=\essinf_F u+ \frac{\essosc_F u}{2^{\lambda+s-1}},
\end{align*}
where $s>1$. By \eqref{linecoldness} we have $(u-k)_-(x,t')=0$ for $\mu$-a.e. $x\in B(x_0,r/2)$, and so by \eqref{DGC limit}, we obtain
\begin{equation}\label{timecontinuation}
\begin{split}
\esssup_{t' < t < t_0} \int_{B(x_0,\frac{r}{4})} &(u-k)_-^2 (x,t)\, d\mu\leq \frac{C}{r^p}\int_{t'}^{t_0} \int_{B(x_0,\frac{r}{2})}   (u-k)_-^p \, d\mu \, dt.
\end{split}
\end{equation}
On the other hand, for each $t\in (t',t_0)$ in the set
\begin{align*}\label{coldset}
\left\{x\in B\l(x_0,r/4\r)\,:\,u(x, t)<\essinf_F u+ \frac{\essosc_F u}{2^{\lambda + s}}\,\right\},
\end{align*}
we have
\begin{align*}
(u-k)_-(x,t)\geq   \frac{\essosc_F u}{2^{\lambda +s}}.
\end{align*} 
We use this with inequality \eqref{timecontinuation} to conclude that for almost every $t\in (t',t_0)$
\begin{align*}
&\left(\frac{\essosc_F u}{2^{\lambda+s}}\right)^2 \mu\l(\l\{B(x_0,r/4):u(\cdot, t)<\essinf_F u+ \frac{\essosc_F u}{2^{\lambda + s}}\,\r\}\r)\\
&\qquad\leq \frac{C}{r^p}\left(\frac{\essosc_F u}{2^{\lambda+s-1}}\right)^p|t_0-t'| \mu\l(B\l(x_0,r/4\r)\r).
\end{align*}
Since 
\begin{align*}
|t_0-t'|< \left(\frac{\essosc_F u}{2^\lambda} \right)^{2-p}2r^p,
\end{align*}
we can further estimate the right hand side to obtain
\begin{align*}
\begin{split}
&\mu\left(\left\{B\left(x_0,r/4\right):\,u(x, t)<\essinf_F u+ \frac{\essosc_F u}{2^{\lambda + s}}\,\right\}\right)\leq \frac{C\mu\left( B\left(x_0,r/4\right)\right)}{2^{s(p-2)}},  
\end{split}
\end{align*}
for almost all $t\in(t', t_0)$, where $C$ depends only on the constant in Lemma \ref{time-independent main lemma}. This implies
\begin{align*}
\frac{\nu(\{\,B(x_0,r/4)\times(t',t_0):\,u<\essinf_F u+ \frac{\essosc_F u}{2^{\lambda + s}}\,\} )}{ \nu(B(x_0,r/4)\times(t',t_0))}\leq \frac{C}{2^{s(p-2)}},
\end{align*}
as desired.
\end{proof}

\end{lemma}

We can now use the above Lemma together with Lemma \ref{time-independent main lemma} to guarantee that the initial conditions needed to employ Lemma \ref{geometric_convergence} are met in a cylinder with upper time level at $t_0$.

\begin{lemma}\label{time_continuation}
Assume the first alternative. Then there exists a positive integer $s$ which depends only on the data and on $\lambda$, such that
\[
u\ge \essinf_F u+\frac{\essosc_F u}{2^{s+1}}\qquad\text{a.e. in}\quad B(x_0,r/8)\times(t_0-\theta_-(r/8)^p,t_0).								
\]

\begin{proof}
In Lemma \ref{time-independent main lemma}, set
\[
r_0=\frac{r}{8}, \quad r_n=\frac{r}{8}+\frac{r}{8\cdot 2^n},
\]
\[
 Q_n^-=B_n\times T=B(x_0,r_n)\times (t',t_0),
\]
and
\begin{align*}
k_n^-=\essinf_F u+\frac{\essosc_F u}{2^{\lambda+s+1}}+\frac{\essosc_F u}{2^{\lambda+s+n+1}},\\
 A_n^-=\left\{(x,t)\in Q_n^-: (u(x,t)-k_n^-)_-> 0\right\},
\end{align*}
where $s\geq 1$. Thus for every $n\geq 0$, we have
\[
 (u-k_n^-)_-\le \eps\essosc_F u, \quad\text{where}\quad\eps=\frac{1}{2^{\lambda+s}},
\]
and for every $n=0,1,\dots$
\begin{align*}
(u-k_n^-)_-(x,t')=0 \textrm{ for }\mu \textrm{ a.e. } x\in B(x_0,r_n).
\end{align*}
By Lemma \ref{initial_estimate} we can now choose $s$ so large that the corresponding initial condition \eqref{strt-engine} 
\begin{align*}
\frac{\nu(A_0^-)}{\nu(Q_0^-)}\leq ((C_0(2+2^{(\lambda+1)(p-2)})^{2-p/\kappa})^{-1/(1-p/\kappa)} (4^{p2(2-p/\kappa)})^{1-(1-p/\kappa)^2}
\end{align*} 
is satisfied.
By Lemma~\ref{geometric_convergence}, we can then conclude that $\nu(A_{2n}^-)/\nu(Q_{2n}^-)\rightarrow 0$ as $n\rightarrow \infty$, which implies that
\begin{align*}
u\ge \essinf_F u+\frac{\essosc_F u}{2^{s+1}}\qquad\text{a.e. in}\quad B(x_0,r/8)\times(t',t_0),
\end{align*}
for some $s$ which depends only on the data and on $\lambda$. Since $t'\leq t_0-\theta_- \left( r/4 \right)^p$, the proof is complete.
\end{proof}
\end{lemma}

Combining the above results, we obtain that the first alternative implies a reduction of oscillation in a subcylinder with upper time level at $t_0$.  
\begin{corollary}\label{reduction of oscillation first alternative}
\noindent Assume that the first alternative holds. Then there exists a constant $\sigma_0\in(0,1)$ which depends only on the data and on $\lambda$, such that
\begin{align*}
 \essosc_{B(x_0,\frac{r}{8})\times (t_0-\theta_-\left( \frac{r}{8} \right)^p, t_0)} u\leq \sigma_0 \essosc_F u.
\end{align*}
 \end{corollary}
\begin{proof}
As a consequence of  Lemma \ref{time_continuation}, we know that there exists an $s\in \mathbb N$, which depends only on the data and on $\lambda$, such that
\begin{align*}
\essosc_{B(x_0,\frac{r}{8})\times (t_0-\theta_-\left( \frac{r}{8} \right)^p, t_0)}u&\leq \esssup_{F} u - \essinf_F u- \frac{\essosc_F u}{2^{s+1}}\\
&=\left(1-\frac{1}{2^{s+1}}\right)\essosc_F u.
\end{align*} 
\end{proof}
This finishes the first alternative.

\section{The Second Alternative}

Suppose that for every $t^*\in(t_0-\theta_+r^p, t_0)$
we have
\begin{equation}\label{assumption2}
\begin{split}
&\nu\left(\left\{B(x_0,r)\times(t^*-\theta_- r^p,t^*) :u(x,t)> k_0^-\right\}\right) \\&\quad\le (1-\alpha_0) \nu(B(x_0,r)\times(t^*-\theta_- r^p,t^*)),
\end{split}
\end{equation}
where $\alpha_0$ is as in the first alternative and $k_0^-$ is as in~\eqref{k_n^-}.  This assumption is called \emph{the second alternative}. Note that the second alternative is exactly the complement of the first alternative.

This alternative is also based on Lemma~\ref{main_lemma}, 
but now we set 
\begin{equation}\label{k_n^+}
 k_n^+=\esssup_{F}{u}-\frac{\essosc_{F}{u}}{2^{\lambda+1}}-\frac{\essosc_{F}{u}}{2^{\lambda+1+n}}
\end{equation}
and will assume $\lambda$ to be large enough so that we can force $\nu(A_0^+)$ to be small compared to $\nu(Q_0^+)$. We start with the following lemma.
\begin{lemma}\label{logarithmic_bound}
Let $u$ belong to the parabolic De Giorgi class and let the second alternative be in force. Let
\[
l=\esssup_{F}{u}-\frac{\essosc_{F}{u}}{2^{2s}}.
\]
Then there exists a positive integer $s$ which depends only on the data such that for almost every $t\in\left(t_0-\theta_+ r^p,t_0\right)$
\[
\mu\left(\left\{x\in B(x_0,r): u(x,t)>l
\right\}\right)
\le\frac{1-\displaystyle{\frac{3\alpha_0}4}}{1-\displaystyle{\frac{\alpha_0}2}}\mu(B(x_0,r)).
\]

\begin{proof}
Let $t^*\in(t_0-\theta_+r^p,t_0)$. The second alternative implies that there exists a time level 
$t'\in (t^*- \theta_- r^p,t^*-\frac{\alpha_0}{2}\theta_- r^p)$
for which 
\begin{equation}\label{second_alternative}
\begin{split}
\mu \left(\left\{ x \in B(x_0,r) : u(x,t')>k_0^-\right\}\right)\le \frac{1-\alpha_0}{1-\displaystyle{\frac{\alpha_0}2}}\mu(B(x_0,r)).
\end{split}
\end{equation}
Indeed, if this was not the case, we would have
\begin{align*}
&\nu\left(\left\{ (x,t) \in B(x_0,r)\times(t^*-\theta_- r^p,t^*) : u(x,t)>k_0^-\right\}\right) \\
&\ge\int_{t^*- \theta_- r^p}^{t^*-\frac{\alpha_0}{2} \theta_- r^p}\mu\left(\left\{ x\in B(x_0,r) : u(x,t) >k_0^-\right\}\right) \, dt \\
&>(1-\alpha_0)\nu(B(x_0,r)\times(t^*-\theta_- r^p,t^*)),
\end{align*}
which contradicts~\eqref{assumption2}. Choose such a $t'$. Let $s$ be a positive integer. We substitute
\[
k=\esssup_{F}{u}-\frac{\essosc_{F}{u}}{2^{s}}
\]
in \eqref{DGC limit} to obtain
\begin{align*}
&(l-k)_+^2\mu(\{x\in B(x_0,(1-\delta)r):u(x,t)>l\})\\
&\le \esssup_{t'<t<t^*}\int_{B(x_0,(1-\delta) r)}(u(x,t)-k)_+^2\, d\mu \\
&\le C\int_{B(x_0,r)}(u(x,t')-k)_+^2\, d\mu+\frac{C}{(\delta r)^p}\int_{B(x_0,r)\times(t^*-\theta_- r^p,t^*)}(u-k)_+^p\, d\nu\\
&\le C\left(\frac{\essosc_{F}{u}}{2^s}\right)^2\mu(\{x\in B(x_0,r):u(x,t')> k\}) \\
&\qquad+\frac{C}{ (\delta r)^p}\left(\frac{\essosc_{F}{u}}{2^s}\right)^p\nu(B(x_0,r)\times(t^*-\theta_- r^p,t^*))
\end{align*}
for almost every $t\in(t',t^*)$ and any $0<\delta<1$. With~\eqref{second_alternative} and the definition of $\theta_-$, this gives
\begin{align*}
&\mu(\{x\in B(x_0,(1-\delta) r):u(x,t)>l\})\\ 
&\le\left(\frac{\essosc_{F}{u}}{(l-k)_+2^s}\right)^2\frac{1-\alpha_0}{1-\displaystyle{\frac{\alpha_0}2}}\mu(B(x_0,r))\\
&\quad+\frac{C}{ (\delta r )^p}\left(\frac{\essosc_{F}{u}}{(l-k)_+2^s}\right)^2 \left(\frac{1}{2^s}\right)^{p-2} r^p \mu(B(x_0,r)) \\
&\le\left(\displaystyle{1-\frac{1}{2^{s}}}\right)^2 \left( \frac{1-\alpha_0}{1-\displaystyle{\frac{\alpha_0}2}}+\frac{C}{ \delta^p}\left(\frac{1}{2^s}\right)^{p-2}\right)\mu(B(x_0,r)),
\end{align*}
for almost every $t\in(t',t^*)$ and any $0<\delta<1$. By the $\alpha$-annular decay property~\eqref{annular_decay}, we have
\begin{align*}
&\mu(\{x\in B(x_0,r):u(x,t)>l\}) \\
&\le\mu(B(x_0,r)\setminus B(x_0,(1-\delta)r))+\mu(\{x \in B(x_0,(1-\delta) r):u(x,t)>l\})\\
&\le C\delta^\alpha\mu(B(x_0,r))+\mu(\{x \in B(x_0,(1-\delta) r):u(x,t)>l\}).
\end{align*}

Hence, by first choosing $\delta$ small enough and after this choosing $s$ large enough, we obtain that for almost every $t\in(t^*-\frac{\alpha_0}{2} \theta_-r^p,t^*)$
\begin{align}\label{lessthanone}
\mu(\{x\in &B(x_0,r):u(x,t)>l\})  \leq \frac{1-\displaystyle{\frac{3\alpha_0}4}}{1-\displaystyle{\frac{\alpha_0}2}}\mu(B(x_0,r)).
\end{align}
The choices of $\delta$ and $s$ depend only on the data.
Finally, since the same choice of $s$ is valid for every $t^*\in(t_0-\theta_+r^p,t_0)$, we conclude that \eqref{lessthanone} holds for almost every $t\in(t_0-\theta_+r^p,t_0)$.
\end{proof}
\end{lemma}

Now we are ready to prove the final lemma which together with Lemma~\ref{geometric_convergence} gives the reduction of oscillation in case of the second alternative. For a constant $k$ define
\[
E_{k}^\tau=\left\{(x,t)\in B(x_0,\tau r)\times (t_0- \theta_+r^p,t_0)
: u(x,t)> k\right\},
\]
where $1\leq \tau \leq 2$. For $\tau=1$ we denote $E_{k}^1=E_{k}$.

\begin{lemma}\label{choosing_lambda}

For every $0<\alpha<1$ there exists a positive constant $\lambda$, which depends only on the data and on $\alpha$, such that
\[
\frac{\nu(E_{k_0^+})}{\nu(B(x_0,r)\times(t_0-\theta_+ r^p,t_0))}\le \alpha,
\]
where as in~\eqref{k_n^+},
\begin{equation*}
 k_0^+=\esssup_{F}{u}-\frac{\essosc_{F}{u}}{2^{\lambda}}.
\end{equation*}

\begin{proof}
Fix $0<\alpha<1$. Define
\[
E_h^\tau(t)=\left\{x\in B(x_0,\tau r)
: u(x,t)> h\right\},
\]
and for constants $h,k$ such that $h>k>k_0^+$, let
\[
v=
\begin{cases}
h-k,\quad &u\ge h, \\
u-k,\quad &k<u<h, \\
0, \quad &u\le k.
\end{cases}
\]
By the previous lemma we can choose $\lambda$ big enough so that for almost every $t\in(t_0-\theta_+r^p,t_0)$ we have
\begin{align*}
\mu(\{x\in B(x_0,r): v(x,t)=0\})&=\mu(\{x\in B(x_0,r): u(x,t)\le k\}) \\
&\ge\mu(\{x\in B(x_0,r): u(x,t)<k_0^+\}) \\
&\ge\frac{\alpha_0}{4}\mu(B(x_0,r)).
\end{align*}
Thus for almost every $t\in(t_0-\theta_+r^p,t_0)$
\[
v_{B(x_0,r)}(t)=\dashint_{B(x_0,r)} v(x,t) \, d\mu \le \l(1-\frac{\alpha_0}{4}\r)(h-k)
\]
and consequently
\[
h-k-v_{B(x_0,r)}(t) \ge \frac{\alpha_0}{4}(h-k).
\]

Using the weak $(q,q)$-Poincar\'e inequality for some $q<p$, see Remark \ref{poincare_remark},  gives

\begin{align*}
(h-k)^q\mu(E_h(t))&\le\l(\frac{4}{\alpha_0}\r)^q\int_{B(x_0,r)}|v(x,t)-v_{B(x_0,r)}(t)|^q\, d\mu \\
&\le
Cr^q\int_{B(x_0,\tau r)\times\{t\}} g_v^q \, d\mu = Cr^q\int_{E_k^\tau(t)\setminus E_h^\tau(t)} g_u^q \, d\mu
\end{align*}
for almost every $t\in(t_0-\theta_+r^p,t_0)$. Next we integrate the above inequality over time to get
\[
(h-k)^q\nu(E_h)\le C r^q\int_{E_k^\tau\setminus E_h^\tau} g_u^q \, d\nu.
\]
Now H\"older's inequality gives
\begin{align*}
(h-k)^q\nu(E_h)&\le Cr^q\l(\int_{E_k^\tau\setminus E_h^\tau} g_u^p \, d\nu\r)^{q/p}\nu(E_k^\tau\setminus E_h^\tau)^{1-q/p} \\
&\le Cr^q\l(\int_{B(x_0,\tau r)\times\l(t_0-\theta_+r^p,t_0\r)} g_{(u-k)_+}^p \, d\nu\r)^{q/p}\nu(E_k^\tau\setminus E_h^\tau)^{1-q/p}.
\end{align*}
Choose
\[
k(s)=\esssup_{F}{u}-\frac{\essosc_{F}{u}}{2^{s}}.
\]
By \eqref{DGC limit} and since $\mu(B(x_0,2r))\theta_+r^p=\nu\left(B(x_0,2r)\times\left(t_0-\theta_+ r^p ,t_0\right)\right)$, we obtain for every $s\leq \lambda$
\begin{equation}\label{nabla_u}
\begin{split}
&\int_{B(x_0,\tau r)\times\l(t_0-\theta_+r^p ,t_0\r)} g_{(u-k)_+}^p \, d\nu\le \int_{B(x_0,2 \tau r)}(u(x,t_0-\theta_+  r^p)-k)_+^2\, d\mu\\
&\qquad+\frac{C}{(\tau r)^p}\int_{B(x_0,2  \tau r)\times(t_0-\theta_+  r^p,t_0)} (u-k)_+^p\, d\nu \\
&\quad\le \frac{C}{ r^p}\left(\frac{\essosc_{F}{u}}{2^s}\right)^p\nu\left(B(x_0,r)\times\left(t_0-\theta_+ r^p ,t_0\right)\right).
\end{split}
\end{equation}
In the last step we also used the doubling property of $\mu$. Choosing now
\[
h(s)=\esssup_{F}{u}-\frac{\essosc_{F}{u}}{2^{s+1}},
\]
yields
\[
\nu(E_{h(s)})\le C 
\nu\left(B(x_0,r)\times\left(t_0- \theta_+ r^p ,t_0\right)\right)^{q/p}\nu(E_{k(s)}^\tau\setminus E_{h(s)}^\tau)^{1-q/p}.
\]
Finally, summing this over $s=1,\dots,\lambda-1$ and then using the doubling condition to replace $\tau r$ by $r$ gives
\begin{align*}
(\lambda-1)\nu(E_{k_0^+})^{p/(p-q)}
\le C\nu\left(B(x_0,r)\times\left(t_0- \theta_+r^p ,t_0\right)\right)^{q/(p-q)}\\
\cdot\nu\left(B(x_0,r)\times\left(t_0- \theta_+ r^p ,t_0\right)\right)
\end{align*}
and hence
\[
\nu(E_{k_0^+})\le \frac{C}{(\lambda-1)^{(p-q)/p}}\nu\left(B(x_0,r)\times\left(t_0- \theta_+ r^p ,t_0\right)\right).
\]
Choosing $\lambda$ large enough finishes the proof.
\end{proof}
\end{lemma}

Now we are in the position to prove the reduction of oscillation in the case of the second alternative, and then complete the proof of the Hölder continuity of $u$.  

\begin{lemma}
Let $u$ belong to the parabolic De Giorgi class and let the second alternative be in force. Then
\[
u\ge \esssup_F u+\frac{\essosc_F u}{2^{\lambda+1}}\qquad\text{a.e. in}\quad B(x_0,r/2)\times\left(t_0-\theta_+\left(r/2\right)^p,t_0\right),
\]
where $\lambda$ depends only on the data.
\begin{proof}
Set
\begin{align*}
 r_n&=\frac{r}{2}+\frac{r}{2^{n+1}}, \qquad r_0=r,\\
Q_n^+&=B_n\times T_n= B(x_0,r_n)\times (t_0-\theta_+ r_n^p, t_0),
\end{align*}
and 
\begin{align*}
 k_n^+&=\esssup_F u-\frac{\essosc_F}{2^{\lambda+1}}-\frac{\essosc_F}{2^{\lambda+n+1}},\\
A_n^+ &=\left\{\,(x,t)\in Q_n^+:\, u(x,t)>k_n^+\right\}.
\end{align*}
Hence the corresponding condition 
Then for every $n\geq 0$, we have
\begin{align*}
(u-k_n^-)_-\leq \eps \essosc_F u,\quad |k_{n+1}^\pm-k_n^\pm|\ge\varepsilon \frac{\essosc_F{u}}{2^{n+2}},\quad\textrm{ where } \quad \varepsilon = \frac{1}{2^\lambda}.
\end{align*}
Hence, the initial condition \eqref{strt-engine} corresponding to Lemma \ref{main_lemma} takes on the form
\begin{align*}
\frac{\nu(A_0^+)}{\nu(Q_0^+)}\leq (C_0(3+1^{p-2}+1^{2-p})^{2-p/\kappa})^{-1/(1-p/\kappa)}(4^{p2(2-p/\kappa)})^{1-(1-p/\kappa)^2}.
\end{align*}
Since the right hand side depends only on the data, by Lemma \ref{choosing_lambda} we see that there exists a $\lambda$, which depends only on the data, for which the above condition is satisfied. Assume $\lambda$ is such. Then as $n \rightarrow 0$ we have $\nu(A_{2n}^+)/\nu(Q_{2n}^+)\rightarrow 0$, which implies the statement of the lemma.
\end{proof}
\end{lemma}

\begin{corollary}\label{reduction of oscillation}
Suppose that the second alternative holds. Then there exists a $\sigma_1\in (0,1)$, which depends only on the data, such that
\[
\essosc_{B(x_0,\frac{r}{2})\times(t_0-\theta_+\left(\frac{r}{2}\right)^p,t_0)}{u}\le \sigma_1 \essosc_{F}{u}.
\]

\begin{proof}
By the previous lemma, we have
\[
\esssup_{B(x_0,\frac{r}{2})\times(t_0-\theta_+\left(\frac{r}{2}\right)^p,t_0)}{u}
\le \esssup_{F}{u}-\frac{\essosc_{F}{u}}{2^{\lambda+1}}
\]
for some $\lambda>1$ which depends only upon the data. 
This implies the statement of the lemma with
\[
\sigma_1=1-\frac{1}{2^{\lambda+1}}.
\]
\end{proof}

\end{corollary}

\noindent \textbf{Proof of Theorem \ref{Holder continuity}}. Since either the first alternative or the second alternative is in force, by Corollary \ref{reduction of oscillation first alternative}  and Corollary \ref{reduction of oscillation} we know that 
\begin{align*}
\essosc_{B(x_0,\frac{r}{8})\times(t_0-\theta_-\left(\frac{r}{8}\right)^p,t_0)}{u}\le \sigma \essosc_{F}{u},
\end{align*}
where $\sigma=\max\{\sigma_0,\sigma_1\}<1$ and $\theta_-$ depend only on the data. The local Hölder continuity of $u$ now follows from this reduction of oscillation, by a standard recursive argument presented for example in \cite{Urba08}, p.44.

\section{Regularity of parabolic quasiminimizers} 

In this section we show that parabolic quasiminimizers belong to the parabolic De Giorgi class. By the previous sections, bounded parabolic quasiminimizers are thus locally Hölder continuous. 

\subsection{Parabolic quasiminimizers}
\begin{definition}
  A function $u \in L_{\textrm{loc}}^p(0,T;N_{\textrm{loc}}^{1,p}(\Omega))$ is a parabolic $K$-quasi-minimizer if there exists a constant $K\ge 1$ and a Carath\'eodory function $G(x,t,u,g)$, satisfying the growth condition
\begin{equation}\label{structure_assumptions}
 C_1|q|^p\le G(x,t,v,q)\le C_2|q|^p
\end{equation}
for some positive constants $C_1$ and $C_2$, such that
\begin{equation} \label{quasi-minimizer}
\begin{split}
-\int\limits_{F} u\frac{\partial\phi}{\partial t}\, d\nu +E(u;F)\le K E(u-\phi;F)
\end{split}
\end{equation}
for every open $F\subset \subset \Omega_T$ and $\phi \in C^{\infty}(0,T;N^{1,p}(\Omega))$ such that $\{\phi\neq 0\}\subset F$. 
Here
\[
 E(u;F)=\int\limits_{F} G(x,t,u,g_u) \, d\nu.
\]
\end{definition}
By \eqref{structure_assumptions} and \eqref{quasi-minimizer} we have that if $u$ is a $K$-quasiminimizer, then there exists positive constants $0<C_1<C_2$ such that
\begin{align}\label{q-mimimum after estimates}
-\int_{F}u\frac{\partial \phi}{\partial t} \, d\nu + &C_1 \int_{F} g_{u}^p \, d\nu\le K C_2 \int_{F} g_{u-\phi}^p \, d\nu,
\end{align}
for every open $F\subset \subset \Omega_T$ and $\phi \in C^{\infty}(0,T;N^{1,p}(\Omega))$ such that $\{\phi\neq 0\}\subset F$. This implies the following.
\begin{lemma}\label{equivalent q-minimum} Let $u \in L_{\textrm{loc}}^p(0,T;N_{\textrm{loc}}^{1,p}(\Omega))$ be a parabolic $K$-quasiminimizer. Then there exist $0<C_1<C_2$ such that
\begin{align*}
-\int_{\{\phi \neq 0\}}u\frac{\partial \phi}{\partial t} \, d\nu + &C_1 \int_{\{\phi \neq 0 \}} g_{u}^p \, d\nu\le K C_2 \int_{\{\phi \neq 0 \}} g_{u-\phi}^p \, d\nu
\end{align*}
for every $\phi \in C^{\infty}(0,T;N^{1,p}(\Omega))$ such that supp$\,\phi\subset \subset \Omega$.

\begin{proof}
Let $\varepsilon>0$ and $\phi \in C^{\infty}(0,T;N^{1,p}(\Omega))$ be such that supp$\,\phi\subset \subset \Omega$. Since $\{\phi \neq 0\}$  is $\nu$-measurable and compactly contained in $\Omega_T$, and since $g_u\in L^p(\Omega_T)$, there exists an open set $F\subset\subset  \Omega_T$ such that 
\begin{align*}
\int_{F\setminus \{\phi\neq 0\}} g_{u}^p\,d\nu \leq \frac{\varepsilon}{KC_2}.
\end{align*}
Also, since $\phi$ is continuous with respect to time we have $\nu(\{\phi=0,\, \partial \phi/\partial t\neq 0)\})=0$, and so by \eqref{q-mimimum after estimates} we can write
\begin{align*}
&-\int_{\{\phi\neq 0\}}u\frac{\partial \phi}{\partial t} \, d\nu + C_1 \int_{\{\phi\neq 0\}} g_{u}^p \, d\nu\leq -\int_{F}u\frac{\partial \phi}{\partial t} \, d\nu + C_1 \int_{F} g_{u}^p \, d\nu\\& K C_2 \int_{F} g_{u-\phi}^p \, d\nu\leq K C_2 \int_{\{\phi\neq 0\}} g_{u-\phi}^p \, d\nu+\varepsilon.
\end{align*}
This holds for every $\varepsilon>0$, which completes the proof.
\end{proof}
\end{lemma}

Our aim in what follows is to prove that a parabolic quasiminimizer $u$ belongs to the parabolic De Giorgi Class, i.e. fulfills the estimate of Definition \ref{De Giorgi class}. 

A fundamental part of the proof is to use partial integration on $u$ with respect to the time variable. However, the time regularity of the function $u\in L_{\textrm{loc}}^p(0,T;N_{\textrm{loc}}^{1,p}(\Omega))$ is a priori not sufficient for this. Therefore we first establish suitable estimates for $u_\varepsilon$, which denotes the standard time mollification of $u$. Having done this we  then pass to the limit as $\varepsilon \rightarrow0$, and obtain the desired results for $u$. 

In the Euclidean setting this argument works, since the theory of mollifiers together with the linearity of taking a gradient guarantees that $u_\varepsilon \rightarrow u$ and $\nabla u_\varepsilon \rightarrow \nabla u$ in $L_{\textrm{loc}}^p(\Omega_T)$ as $\varepsilon \rightarrow 0$. Unfortunately however, the presence of upper gradients in place of usual gradients causes complication to the argument. Since the operation of taking an upper gradient is not linear, it turns out to be problematic to show that $g_{u-u_\varepsilon}\rightarrow 0$ in $L_{\textrm{loc}}^p(\Omega_T)$ as $\varepsilon\rightarrow 0$. It would be interesting to know whether or not it is possible to show this using only the theory of upper gradients. 

We circumvent this question by using the known comparability between upper gradients and so called Cheeger derivatives. As will be seen in the following, the Cheeger derivative has the property of being a linear operation. 

\subsection{The Cheeger derivative}

The following theorem, which yields in a local sense the notion of partial derivatives in metric space, is by Cheeger \cite{Chee99}. For a concise source of tools given by the theory of Cheeger derivatives we refer to \cite{BjorBjorShan03} and the references therein.

\begin{theorem}
Let $X$ be a metric measure space equipped with a positive doubling Borel regular measure $\mu$. Assume $X$ admits a weak $(1,p)$-Poincaré inequality for some $1<p<\infty$.

Then there exists a countable collection $(U_\alpha, X^\alpha)$ of measurable sets $U_\alpha$ and Lipschitz functions $X^\alpha=(X_1^\alpha, \dots,X_{k(\alpha)}^\alpha):X\rightarrow \R^{k(\alpha)}$ such that $\mu \left(X\setminus \bigcup_\alpha U_\alpha \right)=0$ and for all $\alpha$, the following hold:

The functions $X_1^\alpha,\dots, X_{k(\alpha)}^\alpha$ are linearly independent on $U_\alpha$ and $1\leq k(\alpha)\leq N$, where $N$ is a constant depending only on the doubling constant of $\mu$ and the constant from the Poincaré inequality. If $f:X\rightarrow \R$ is Lipschitz, then there exist unique measurable bounded vector valued functions $d^\alpha f: U_\alpha \rightarrow \R^{k(\alpha)}$ such that for $\mu$-a.e. $x_0 \in U_\alpha$,
\begin{align*}
\lim_{r\rightarrow 0+} \sup_{x\in B(x_0,r)} \frac{\left|f(x)-f(x_0)-\left<d^\alpha f(x_0), X^\alpha(x)-X^\alpha (x_0)\right>\right|}{r}=0.
\end{align*}

\end{theorem}
A non negative function $|\cdot|_{1,x_0}$ is introduced \cite{Chee99}, p.460, on $d_\alpha f(x_0)$ such that 
\begin{align*}
 |d^\alpha f(x_0)|_{1,x_0}=g_f(x_0),
\end{align*}
where $g_f$ is the minimal $p$-weak upper gradient of $f$.
Furthermore, it is shown that one can find an inner product norm\\ $|\cdot |_{x_0}:\R^{k(\alpha)}\rightarrow [0,\infty)$ which is $C$-quasi-isometric to $|\cdot|_{1,x_0}$, where the constant $C$ depends only on $k(\alpha)$.

We may assume that the sets $U_\alpha$ are pairwise disjoint. For each $\alpha$, extend $d^\alpha f$ to be zero in the set $X\setminus U_\alpha$, and define 
\begin{align*}
Df: X \rightarrow \R^{k(\alpha)}, \\
Df=\sum_\alpha  d^\alpha f.
\end{align*}
The above imply that the differential mapping $D:f \mapsto Df$ is linear, and that there is a constant $C>0$, which depends only on $N$, such that for all Lipschitz functions $f$ and $\mu$-a.e. $x\in X$

\begin{align}\label{comparability2}
 \frac{1}{C}|Df(x)|\leq g_f(x)\leq C|Df(x)|,
\end{align}
where by $|Df(x)|$ we mean $|d^{\alpha} f(x)|_x$, whenever $x\in U_\alpha$.

From \cite{Shan00} it is known  that the Newtonian space $N^{1,p}(X)$ is the closure in the $N^{1,p}$-norm, of the collection of Lipschitz functions on $X$ with finite $N^{1,p}$-norm. By \cite{FranHajlKosk99}  we know that there exists a unique gradient $Du$ which satisfies \eqref{comparability2} for every $u\in N^{1,p}(X)$. Also, if $\{u_j\}_{j=1}^\infty$ is a sequence in $N^{1,p}(X)$, then $u_j\rightarrow u$ in $N^{1,p}(X)$  if and only if $u_j \rightarrow u$ and $Du_j \rightarrow Du$ in $L^p(X,\mu; \R^N)$, as $j\rightarrow \infty$.

Analogously to what was done with upper gradients, we define the parabolic Cheeger derivative of a time dependent function by taking the Cheeger derivative with respect to the variable $x$, at time level $t$.  

Next we prove the steps which will be used to overcome the complications in the mollification argument, caused by the non linearity of upper gradients. Here and in what follows we denote by $u_h$ the time mollification of a function $u$, i.e.
\begin{align*}
 u_h(x,t)= \int_{-h}^h \eta_h(s) u(x,t-s)\,ds,
\end{align*}
where $\eta_h(s)=h^{-1} \eta(s/h)$ denotes a standard mollifier.

\begin{lemma}\label{convergence of upper gradient}
 Assume $u\in L_{\textrm{loc}}^p(0,T;N_{\textrm{loc}}^{1,p}(\Omega))$. Then as $h\rightarrow 0$, it holds 
$g_{u-u_h} \rightarrow 0$ in $L_{\textrm{loc}}^p(\Omega_T)$, and also pointwise $\nu$-almost everywhere in $\Omega_T$. Moreover, as $s \rightarrow 0$,  we have $g_{u(\cdot, \cdot-s)-u} \rightarrow 0$ in $L_\textrm{loc}^p(\Omega_T)$.
\begin{proof}
Let $h>0$. For $\nu$-almost every $x,y\in \Omega$, $t\in [t_1+h,t_2-h]$ and every compact rectifiable path $\gamma$ from $x$ to $y$, we have
\begin{align*}
 |u_h(x,t)-u_h(y,t)|&\leq \int_{t_1}^{t_2} \left(\int_\gamma (g_u)(z,s)\,dz \right)\eta_h(t-s)\,ds\\
&=\int_{\gamma} (g_u)_h(z,t)\,dz.
\end{align*}
Hence $(g_u)_h$ is a $p$-weak upper gradient of $u_h$. The definition of the minimal $p$-weak upper gradient now implies that for $\nu$-almost every $(x,t)\in\Omega\times [t_1+h,t_2-h]$ 
 \begin{align}\label{mollification is an upper gradient}
g_{u_h}(x,t) \leq (g_u)_h(x,t).
\end{align}

We now show that for $\nu$-every $(x,t)\in \Omega \times (h,T-h)$ we have $D u_h(x,t)=(D u)_h(x,t)$. Assume first that $u\in \textrm{Lip}(\Omega \times (0,T))$. Assume a point $(x_0,t)\in \Omega\times (0,T)$ where with respect to the spatial variable, the Cheeger derivative of $u$ exists. Let $r>0$ be such that $B(x_0,r)\subset \Omega$. Then, since $u$ and $X^\alpha$ are Lipschitz-continuous with modulus $C_{\textrm{Lip}(u)}$ and $C_{\textrm{Lip}(X^\alpha)}$ respectfully, we may write for any $x\in B(x_0,r)$
\begin{align*}
 &r^{-1}|u_h(x,t)-u_h(x_0,t)-\left< (d^\alpha u)_h(x_0,t), X^\alpha (x)-X^\alpha(x_0)\right>|\\
\leq& \int_{-h}^{h}\eta_h(t-s)r^{-1} 
|u(x,s)-u(x_0,s)-\left< (d^\alpha u)(x_0,s), X^\alpha (x)-X^\alpha(x_0)\right>|\, ds\\
\leq&\, C\cdot C_{\textrm{Lip}(u)}+C\cdot C_{\textrm{Lip}(u)}\cdot C_{\textrm{Lip}(X^\alpha)}.
\end{align*}
Therefore by the Lebesgue theorem of dominated convergence we have
\begin{align*}
& \lim_{r\rightarrow 0+} \sup_{x\in B(x_0,r)} r^{-1}\left|u_h(x,t)-u_h(x_0,t)-\left<(d^\alpha u)_h(x_0,t), X^\alpha(x)-X^\alpha (x_0)\right>\right|\\
&\leq \int_{-h}^h \eta_h(t-s) \\
&\cdot\lim_{r\rightarrow 0+}\sup_{x\in B(x_0,r)} r^{-1} |u(x,s)-u(x_0,s)-\left< (d^\alpha u)(x_0,s), X^\alpha (x)-X^\alpha(x_0)\right>| \,ds\\
&=0.
\end{align*}
By the uniqueness of the Cheeger derivative, and by the definition of $Du$ the above implies that $D u_h(x_0,t)=(D u)_h(x_0,t)$. Assume then that $u\in L_{\textrm{loc}}^p(0,T;N_{\textrm{loc}}^{1,p}(\Omega))$, not necessarily Lipschitz, and let $F$ be a compact subset of $\Omega_T$. Let $\{u_j\}\subset \textrm{Lip}(\Omega_T)$ be a sequence such that $u_j \rightarrow u$ in $L^p(0,T;N^{1,p}(\Omega))$. 
Then, since $u_j$ is Lipschitz, by inequality \eqref{comparability2} and by \eqref{mollification is an upper gradient}
\begin{align*}
 &\|Du_h-(Du)_h\|_{L^p(F)}\leq \| D(u-u_j)_h\|_{L^p(F)}+\| D(u_j)_h-(Du)_h\|_{L^p(F)}\\
&\leq C\| g_{(u-u_j)_h}\|_{L^p(F)}+\|(Du_j)_h-(Du)_h\|_{L^p(F)}\\
&\leq C\|(g_{u-u_j})_h\|_{L^p(F)}+\|(Du_j-Du)_h\|_{L^p(F)}\\
&\leq C\|g_{u-u_j}\|_{L^p(F)}+C\|Du_j-Du\|_{L^p(F)}.
\end{align*}
Since the last expression tends to zero as $j\rightarrow \infty$, we can conclude that $Du_h(x,t)=(Du)_h(x,t)$ for $\nu$-almost every $(x,t)\in \Omega\times (h,T-h)$.

By inequality \eqref{comparability2}, by the linearity of the Cheeger derivation and since $Du_h=(Du)_h$, we can write for $\nu$-almost every $(x,t)\in \Omega \times (h,T-h)$ 
\begin{align}\label{dominated upper gradient}
 g_{u-u_h} \leq C|Du-(Du)_h|.
\end{align}
Since $g_u\in L_{\textrm{loc}}^p(\Omega_T)$, by \eqref{comparability2}, also $Du \in L_{\textrm{loc}}^p(\Omega_T)$. This means, by the theory of mollifiers, that as $h\rightarrow 0$, on the right side of \eqref{dominated upper gradient} we have convergence to zero in $L_{\textrm{loc}}^p(\Omega_T)$ and also pointwise $\nu$-almost everywhere in $\Omega_T$.  Lastly, for an $s>0$ small enough, by inequality \eqref{comparability2} we have
\begin{align*}
 \int_{F} g_{u(\cdot, \cdot-s)-u}^p \, d\nu\leq C \int_{F} |Du(\cdot,\cdot-s)-Du|^p\, d\nu.
\end{align*}
Again by \eqref{comparability2}, we know that $Du\in L_{\textrm{loc}}^p(\Omega_T)$. Since $F$ is compact, the Lemma now follows from the continuity of the translation operation for $L^p$ functions. 
\end{proof}
\end{lemma}

Now we are set to prove that parabolic quasiminimizers belong to the De Giorgi class.

\begin{theorem}
Let $u\in L_{\textrm{loc}}^p(0,T;N_{\textrm{loc}}^{1,p}(\Omega))$ be a parabolic K-quasimini-\\mizer. Then $u$ belongs to the parabolic De Giorgi class.

\begin{proof}

Let $u$ be a parabolic quasiminimizer. By making a change of variable, it is straightforward to check that if $u(x,t)$ fulfills \eqref{q-mimimum after estimates}, then for any small enough $s$ also $u(x,t-s)$ fulfills \eqref{q-mimimum after estimates}. Assume a function $\phi \in L^p(0,T; N^{1,p}(\Omega))$ such that supp$\,\phi \subset \subset \Omega_T$. Then there exists an $h_0>0$ such that for every $0<h<h_0$,  $\phi_h \in C^\infty(0,T;N^{1,p}(\Omega))$ with supp$\,\phi_h\subset \subset \Omega_T$, and so is a permissible test function.  By Lemma \ref{equivalent q-minimum}, we have
\begin{equation}\label{q-minimum with mollified test function}
\begin{split}
-\int_{\{\phi_h \neq 0\}}u(x,t-s)\frac{\partial \phi_h}{\partial t} \, d\nu + &C_1 \int_{\{\phi_h \neq 0\}} g_{u(\cdot,\cdot-s)}^p \, d\nu,\\
&\le K C_2 \int_{\{\phi_h \neq 0\}} g_{u(\cdot,\cdot-s)-\phi_h}^p \, d\nu.
\end{split}
\end{equation}
We multiply both sides of \eqref{q-minimum with mollified test function} by a standard mollifier with respect to the variable $s$ and which has support $(-\varepsilon, \varepsilon)$. Integrating the resulting expression in the variable $s$ yields, after using Fubini's theorem,
\begin{align*}
-\int_{\{\phi_h \neq 0\}} u_\varepsilon(x,t)\frac{\partial\phi_h}{\partial t}\, d\nu + &C_1 \int_{\{\phi_h \neq 0\}} (g_{u(\cdot,\cdot-s)}^p)_\varepsilon \, d\nu\\
\notag &\le K C_2 \int_{\{\phi_h \neq 0\}} (g_{u(\cdot,\cdot-s)-\phi_h}^p)_\varepsilon \, d\nu.
\end{align*}
Conducting partial integration on the time derivative term and using the triangle inequality for upper gradients on the right side yields
\begin{equation}\label{sidewise integration with h}
\begin{split}
&\int_{\{\phi_h \neq 0\}} \frac{\partial u_\varepsilon}{\partial t} \phi_h \, d\nu + C_1 \int_{\{\phi_h \neq 0\}} ( g_{u(\cdot,\cdot-s)}^p )_\varepsilon\, d\nu \\
 &\le K C \int_{\{\phi_h \neq 0\}} ( g_{u(\cdot,\cdot-s)-\phi}^p )_\varepsilon \, d\nu +KC \int_{\{\phi_h \neq 0\}} g_{\phi-\phi_h}^p \, d\nu .
\end{split}
\end{equation}
As $h\rightarrow 0$, from the theory of mollifiers, it follows that  $\{\phi_h \neq 0\}$ converges to $\{\phi \neq 0\}$ in $\nu$-measure as $h\rightarrow 0$. By Lemma \ref{convergence of upper gradient} the last term on the right hand side of \eqref{sidewise integration with h} converges to zero as $h \rightarrow 0$. Hence after taking the limit $h\rightarrow 0$  we have
\begin{equation}\label{sidewise integration}
\begin{split}
\int_{\{\phi \neq 0\}} \frac{\partial u_\varepsilon}{\partial t} \phi \, d\nu + &C_1 \int_{\{\phi \neq 0\}} (g_{u(\cdot,\cdot-s)}^p)_\varepsilon \, d\nu \\
&\le K C \int_{\{\phi \neq 0\}} (g_{u(\cdot,\cdot-s)-\phi}^p)_\varepsilon  \, d\nu.
\end{split}
\end{equation}

Assume now $x\in \Omega$ and $r_1<r_2$ and $\tau_2<\tau_1<\tau_0$ to be such that $B(x,r_2)\times (\tau_2,\tau_0)\subset \Omega_T$. Let $\chi_{[\tau_2, t]}$ denote the characteristic function of the time interval $[\tau_2, t]$. Let $\varphi \in \textrm{Lip}(\Omega_T)$ be such that $0\leq\varphi \leq 1$, $\varphi=1$ in $B(x,r_1)\times (\tau_1,\tau_0)$, that $\varphi(x,t)=0$ whenever $t\leq\tau_2$ or $x\not \in B(x,r_2)$, and
\begin{align*}
g_\varphi^p\leq \frac{1}{(r_2-r_1)^p} \textrm{ and }\left|\frac{\partial \varphi}{\partial t}\right|\leq \frac{1}{\tau_1-\tau_2}.
\end{align*}
Choose the test function
\[
\phi=\pm\varphi(u_\varepsilon-k)_\pm\chi_{[\tau_2,t]},
\]
where $t$ is arbitrarily fixed in $(\tau_2,\tau_0)$. We can now estimate on the right hand side of \eqref{sidewise integration}, 
\begin{equation*}\label{siloitusarvio}
\begin{split}
 &\int_{\{\phi \neq 0\}} (g_{u(\cdot,\cdot-s)\mp\varphi(u_\varepsilon-k)_\pm \chi_{[\tau_2,t]}}^p)_\varepsilon \, d\nu
\leq \left(C\int_{\{\phi \neq 0\}} g_{u(\cdot,\cdot-s)-u}^p\, d\nu\right)_\varepsilon \\
&+ C\int_{\{\phi \neq 0\}} g_{u-u_\varepsilon }^p \, d\nu+C\int_{\{\phi \neq 0\}} g_{u_\eps-\varphi(u_\eps-k) }^p \, d\nu.
\end{split}
\end{equation*}

By Lemma \ref{convergence of upper gradient}, the first and second terms on the right hand side of the above expression converge to zero as $\varepsilon \rightarrow 0$. For the third term we  write
\begin{align*}
&\int_{\{\phi \neq 0\}} g_{u_\varepsilon-\varphi(u_\varepsilon-k)}^p\,d\nu=\int_{\{\phi \neq 0\}} g_{(1-\varphi)(u_\varepsilon-k)}^p\,d\nu\leq C\int_{\{\phi \neq 0\}} (1-\varphi)^pg_{u_\varepsilon-u}^p\,d\nu\\
&\qquad+C\int_{\{\phi \neq 0\}} (1-\varphi)^pg_{u}^p\,d\nu+C\int_{\{\phi \neq 0\}}(u_\varepsilon-k)_\pm^pg_{\varphi}^p\, d\nu
\end{align*}
Hence, as $\varepsilon \rightarrow 0$, after recalling the properties of $\varphi$, and the definition of $\phi$, we obtain for the right hand side of \eqref{sidewise integration},
\begin{align*}
\limsup_{\varepsilon \rightarrow 0} \int_{\{\phi \neq 0\}} (g_{u(\cdot,\cdot-s)-\phi}^p)_\varepsilon  \, d\nu \leq &C\int_{\tau_2}^{\tau_1}\int_{B(x,r_2)\setminus B(x,r_1)} g_{(u-k)_\pm}^p\,d\nu\\
&+\frac{C}{(r_2-r_1)^p}\int_{\tau_2}^{\tau_0}\int_{B(x,r_2)}(u-k)_\pm^p\, d\nu.
\end{align*}
On the left hand side of \eqref{sidewise integration}, for the first term we have 
\begin{equation*}\label{partintegr}
\begin{split}
&\pm \int_{\{\phi \neq 0\}} \frac{\partial u_\varepsilon}{\partial t} \varphi (u_\varepsilon-k) \chi_{[\tau_2,t]}\, d\nu = \frac{1}{2}\int_{\{\phi \neq 0\}} \frac{\partial}{\partial t} (\varphi (u_\varepsilon-k)_\pm^2 )\chi_{[\tau_2,t]}\, d\nu\\
 & \qquad- 2\int_{\{\phi \neq 0\}} \frac{\partial \varphi}{\partial t} (u_\varepsilon-k)_\pm^2 \chi_{[\tau_2,t]}\, d\nu\\
&\longrightarrow  \frac{1}{2}\int_{B(x,r_2)} \varphi(u(x,t)-k)_\pm^2\,d\mu- 2\int_{\tau_2}^{t}\int_{B(x,r_2)} \frac{\partial \varphi}{\partial t} (u-k)_\pm^2\, d\nu,
\end{split}
\end{equation*} 
as $\varepsilon \rightarrow 0$.
For the second term on the left hand side of \eqref{sidewise integration}, we clearly have
\begin{equation*}
\begin{split} 
\lim_{\varepsilon\rightarrow 0}\int_{\{\phi \neq 0\}} (g_{u(\cdot,\cdot-s)}^p)_\varepsilon \, d\nu 
\geq \int_{\tau_2}^{t}\int_{B(x,r_1)} g_{(u-k)_\pm}^p\,d\nu.
\end{split}
\end{equation*}
Collecting the results, since the constants of the obtained inequality are independent of $t$, we obtain the estimate
\begin{equation*}\label{cleanedestimate}
\begin{split}
&\esssup_{\tau_1<t<\tau_0}\int_{B(x,r_1)} (u(x,t)-k)_\pm^2\,d\mu +\int_{\tau_2}^{\tau_0}\int_{B(x,r_1)} g_{(u-k)_\pm}^p\,d\nu\\
&\leq C\int_{\tau_2}^{\tau_0}\int_{B(x,r_2)\setminus B(x,r_1)} g_{(u-k)_\pm}^p\,d\nu+\frac{C}{(r_2-r_1)^p}\int_{\tau_2}^{\tau_0}\int_{B(x,r_2)}(u-k)_\pm^p\, d\nu\\
&\qquad+\frac{C}{\tau_1-\tau_2}\int_{\tau_2}^{\tau_1}\int_{B(x,r_2)}(u-k)_\pm^2\, d\nu.
\end{split}
\end{equation*}
We then multiply the second term on the left hand side by the constant $C$, and sum the resulting term on both sides of the above inequality, to obtain
\begin{equation}\label{second to last}
\begin{split}
&\esssup_{\tau_1<t<\tau_0}\int_{B(x,r_1)} (u(x,t)-k)_\pm^2\,d\mu +C_1\int_{\tau_1}^{\tau_0}\int_{B(x,r_1)} g_{(u-k)_\pm}^p\,d\nu\\
&\leq C_2\int_{\tau_2}^{\tau_0} \int_{B(x,r_2)}g_{(u-k)_\pm}^p\,d\nu+\frac{C_2}{(r_2-r_1)^p}\int_{\tau_2}^{\tau_0}\int_{B(x,r_2)}(u-k)_\pm^p\, d\nu\\
&\qquad+\frac{C_2}{\tau_1-\tau_2}\int_{\tau_2}^{\tau_1}\int_{B(x,r_2)}(u-k)_\pm^2\, d\nu,
\end{split}
\end{equation}
where $C_1>C_2>0$. Set now
\begin{align*}
f(\rho,\tau)=\frac{1}{C_1}\esssup_{\tau<t<\tau_0}\int_{B(x,\rho)} (u(x,t)-k)_\pm^2\,d\mu+\int_{\tau}^{\tau_0}\int_{B(x,\rho)} g_{(u-k)_\pm}^p\,d\nu.
\end{align*}
Then by \eqref{second to last}, for every $0<r_1<r_2<r$ and $0<\tau_2<\tau_1<\tau_0$,  we have
\begin{align*}
&f(r_1,\tau_1)\leq \frac{C_2}{C_1}f(r_2,\tau_2)\\
&+\frac{C}{(r_2-r_1)^p}\int_{\tau_2}^{\tau_0}\int_{B(x,r_2)}(u-k)_\pm^p\, d\nu+
\frac{C}{(\tau_1-\tau_2)}\int_{\tau_2}^{\tau_1}\int_{B(x,r_2)}(u-k)_\pm^2\, d\nu.
\end{align*}
Using the following real analytic lemma now completes the proof.
\end{proof}
\end{theorem}

\begin{lemma}\label{continuous estimate}
Let $f(\rho,\tau)$ be a nonnegative and bounded function on $[0,r]\times[0, \tau_0]$. If for every $0< r_1 < r_2<r$ and $0 < \tau_2<\tau_1<\tau_0$, we have
\begin{align*}
f(r_1,\tau_1)\leq \sigma_3 f(r_2,\tau_2)+A(r_2-r_1)^{-\alpha}+B(\tau_1-\tau_2)^{-\beta},
\end{align*}
where $A,B,\alpha, \beta, \sigma_3$ are nonnegative constants and $0\leq \sigma_3 <1$, then for every $0< r_1< r_2<r$ and $0 <\tau_2<\tau_1<\tau_0$
\begin{align*}
f(r_1,\tau_1) \leq C A(r_2- r_1)^{-\alpha} +CB(\tau_1-\tau_2)^{-\beta},
\end{align*}
where the constant $C$ depends only on $\alpha,\beta, \sigma_3$.
\begin{proof}
For proof use Lemma 2.1.4 from \cite{WuZhaoYinLi01} for both variables, one at a time.
\end{proof}
\end{lemma}

\def\cprime{$'$} \def\cprime{$'$}

\end{document}